\documentclass[12pt,twoside,reqno]{amsart}
\usepackage[all]{xy}   
\usepackage {amssymb,latexsym,amsthm,amsmath}
\usepackage[colorlinks, linkcolor=red, citecolor=blue]{hyperref}
\usepackage{enumitem}
\usepackage{todonotes}

\usepackage{tikz}
 \usetikzlibrary{matrix,arrows}
        \usepackage[utf8]{inputenc}
        \usepackage{xcolor}
        
\long\def\symbolfootnote[#1]#2{\begingroup%
\def\thefootnote{\fnsymbol{footnote}}\footnote[#1]{#2}\endgroup}
\newcommand{\Z}{\ensuremath{\mathbb{Z}}}

\newcommand{\Aut}{\textup{Aut}}

\newcommand{\R}{\mathbb R}
\newcommand{\N}{\mathbb{N}}

\def \Out {\mathrm{Out}}
\def \MCG {\mathrm{MCG}}
\def \PMCG {\mathrm{PMCG}}
\def \Homeo {\mathrm{Homeo}}
\def \Ends {\mathrm{Ends}}
\def \Inn {\mathrm{Inn}}
\newcommand{\Id}{\mathrm{Id}}

\makeatletter
\def\imod#1{\allowbreak\mkern10mu({\operator@font mod}\,\,#1)}
\makeatother
\newtheorem{theorem}{Theorem}[section]
\newtheorem{lemma}[theorem]{Lemma}
\newtheorem{corollary}[theorem]{Corollary}

\newtheorem*{theorem*}{Theorem}

\theoremstyle{definition}
\newtheorem{remark}[theorem]{Remark}

\newtheorem{example}[theorem]{Example}

\numberwithin{equation}{section}

\newcommand{\ignore}[1]{}

\newcommand{\mynote}[1]{}
\newcommand{\secref}[1]{Section~\ref{#1}}
\newcommand{\thmref}[1]{Theorem~\ref{#1}}
\newcommand{\lemref}[1]{Lemma~\ref{#1}}

\newcommand{\corref}[1]{Corollary~\ref{#1}}

\begin{document}
\setcounter{section}{0}
\setcounter{tocdepth}{1}
\title{Twisted Conjugacy in Big Mapping Class Groups}
\author[Sushil Bhunia]{Sushil Bhunia}
\author[Swathi Krishna]{Swathi Krishna}
\address{Department of Mathematics, BITS-Pilani, Hyderabad Campus, Hyderabad, India}
\email{sushilbhunia@gmail.com}
\address{Tata Institute of Fundamental Research, Homi Bhabha Road, Mumbai, Maharashtra, India}
\email{swathi280491@gmail.com}
\date{\today}
\begin{abstract}
Let $G$ be a group and $\varphi$ be an automorphism of $G$. Two elements $x, y$ of $G$ are said to be $\varphi$-twisted conjugate if $y=gx\varphi(g)^{-1}$ for some $g\in G$. A group $G$ has the $R_{\infty}$-property if the number of $\varphi$-twisted conjugacy classes is infinite for every automorphism $\varphi$ of $G$.
In this paper, we prove that the big mapping class group $\MCG(S)$ possesses the $R_{\infty}$-property under some suitable conditions on the infinite-type surface $S$.
As an application, we also prove that the big mapping class group possesses the $R_\infty$-property if and only if it satisfies the $S_{\infty}$-property.
\end{abstract}
\maketitle
\section{Introduction}\label{intro}
A natural generalization of conjugacy is twisted conjugacy and the study of this can be dated back to $1939$ by the work of Gantmacher (see \cite{gant}). 
Let $G$ be a group and $\varphi$ an automorphism of $G$. Two elements $x, y\in G$ are said to be $\varphi$-twisted conjugate, denoted by $x\sim_{\varphi} y$, if $y=gx\varphi(g)^{-1}$ for some $g\in G$. Clearly, $\sim_{\varphi}$ is an equivalence relation on $G$. The equivalence classes with respect to this relation are called the \textit{$\varphi$-twisted conjugacy classes} or \textit{the Reidemeister classes} of $\varphi$. If $\varphi=\Id$, then the $\varphi$-twisted conjugacy classes are the (ordinary) conjugacy classes. The $\varphi$-twisted conjugacy class containing $x\in G$ is denoted by $[x]_{\varphi}$ or simply $[x]$, when $\varphi$ is clear from the context. The Reidemeister number of $\varphi$, denoted by $R(\varphi)$, is the number of all $\varphi$-twisted conjugacy classes. 
This number is either a natural number or $\infty$.
A group $G$ is said to have the \textit{$R_{\infty}$-property} if $R(\varphi)$ is infinite for every automorphism $\varphi$ of $G$. 

The Reidemeister number is closely related to the Nielsen number of a self-map of a manifold. More precisely, let $M$ be a closed connected manifold of dimension at least $3$ and $f: M\to M$ a continuous map, the minimal number of fixed points among all continuous maps  homotopic to $f$ is the Nielsen number $N(f)$. Note that the Nielsen number $N(f)$ is always finite and is homotopy invariant. The computation of $N(f)$ is difficult and a central object to study in the Nielsen-Reidemeister fixed point theory (cf. \cite{fel10}). This number is bounded above by the Reidemeister number $R(f_{\#})$, where $f_{\#} $ is the induced map on $ \pi_1(M) $. If $M$ is a Jiang-type space, then either $N(f)=0$ or $N(f)=R(f_{\#})$ (we refer the reader to \cite[p. 37, Theorem 5.6]{Jiang} for details). Thus, in this scenario, if 
$R(f_{\#})=\infty$ then $N(f)=0$ which in turn implies that one can deform $f$, up to homotopy, to a fixed-point free map. 

The class of groups which have the $R_{\infty}$-property is a robust and active area of research begun by Fel'shtyn and Hill in \cite{fh94}. The problem has a long history of references, for example, see \cite{fel10, ft15, dlo, ds} and the references therein. 
In the context of algebraic groups, Steinberg proved that
for a connected linear algebraic group $G$ over an algebraically closed field $k$ and a surjective endomorphism $\varphi$ of
G, if the fixed-point subgroup $G^{\varphi}$ is finite then $R(\varphi)=1$, and hence $G$ does not satisfy the $R_{\infty}$-property (see \cite[Theorem 10.1]{St}). The above result is classically known as the \emph{Lang-Steinberg theorem}. Recently, Bose and the first author of this paper proved that a connected linear algebraic group $G$ has the (algebraic) $R_{\infty}$-property if and only if the fixed-point subgroup $G^{\varphi}$ is infinite for every automorphism $\varphi$ of $G$ (cf. \cite[Theorem 17]{bb20}, \cite[Theorem 2.4]{bb21}).
In the realm of geometric group theory, in \cite{ll00, fel01} the authors proved that the non-elementary (Gromov) hyperbolic groups have the $R_{\infty}$-property. Further, non-elementary relatively hyperbolic groups were shown to have this property by Fel'shtyn (see \cite[Theorem 3.3 (2)]{fel10}).

The motivation of the present work comes from the results of Fel'shtyn and Gon\c{c}alves \cite{fg10}. 
The \textit{mapping class group} $\MCG(S)$ of  a surface $S$ is the group of isotopy classes of orientation-preserving homeomorphisms of $S$.
They proved that the mapping class group $\MCG(S)$ of an orientable closed surface $S$ has the $R_{\infty}$-property if and only if $S$ is not $\mathbb{S}^2$ (see \cite[Theorem 4.3]{fg10}). In the Appendix of that paper, Dahmani et al. proved that, except for a few low complexity cases, the mapping class group of a compact surface has the $R_{\infty}$-property,  using geometric group theoretic methods. 
Recently, there has been a surge of activities on infinite-type surfaces and corresponding groups, namely the big mapping class groups (see \cite{bigsurvey}). 

The goal of this paper is to exhibit a similarity between the finite-type and infinite-type settings. 
Now, let $S$ be an infinite-type surface (i.e. $S$ is a connected, second countable, oriented $2$-manifold without boundary such that the fundamental group $\pi_1(S)$ is not finitely generated) and $\MCG(S)$ denote the big mapping class group of $S$.   
As opposed to the above mentioned results (mostly algebraic in nature) our approach will be significantly more geometric. The technique from the Appendix of \cite{fg10} can be modified suitably in our set-up to show that the big mapping class group $\MCG(S)$ does satisfy the $R_{\infty}$-property under some restrictions. 
A finite-type connected subsurface $\Sigma\subset S$ is called \textit{non-displaceable} if $f(\Sigma)\cap \Sigma\neq \emptyset$ for all $ f\in \Homeo(S)$. One of the main results of this paper is the following:

\begin{theorem*}[\thmref{mainthmnondis}]
	If $S$ is a connected orientable infinite-type surface without boundary containing a non-displaceable subsurface, then $\MCG(S)$ has the $R_{\infty}$-property.
\end{theorem*}
\noindent
Using the Birman exact sequence, we prove that if $S$ is an infinite-type surface with an isolated puncture, then the big mapping class group $\MCG(S)$ possesses the $R_{\infty}$-property (see \corref{stosx}). For details, see \secref{mcg} and \secref{results}.
We also prove the following.
\begin{theorem*}[\thmref{111}]
	Let $S$ be an infinite-type surface such that the space of ends $\Ends(S)$ is a Cantor set $C$ with the following: either (a) $\mathrm{genus}(S)<\infty$ or (b) $\Ends(S)=\Ends^{np}(S)$.
	Then $\MCG(S)$ possesses the $R_{\infty}$-property.
\end{theorem*}
As an immediate application of the above results, we characterize the $S_{\infty}$-property of the mapping class groups. For necessary definitions and results, we refer to \secref{isogred}. In particular, we prove the following result in this direction.

\begin{theorem*}[\thmref{mainthm2}]
	Let $S$ be a connected orientable surface without boundary. Then $\MCG(S)$ has the $R_{\infty}$-property if and only if $\MCG(S)$ has the $S_{\infty}$-property.
\end{theorem*}

\section{Preliminaries}\label{prel}
\noindent
In this section, we fix some notations and terminologies and collect some basic results which will be used throughout this paper.
\subsection{Infinite-type surfaces}
Let $S$ be a connected, orientable surface ($2$-dimensional manifold) without boundary. Then
$S$ is said to be of \textit{finite-type} if its fundamental group is finitely generated; otherwise, $S$ is said to be of \textit{infinite-type}. 
Note that whenever the surface $S$ is of infinite-type the fundamental group $\pi_1(S)$ of $S$ is a free group of countable infinite rank (see, for instance, \cite[p. 142, Theorem 4.2.2]{sj93}). Finite-type surfaces (without boundary) are classified by the genus and number of punctures. We need more information to classify infinite-type surfaces, which we will discuss here. Let $S$ be an infinite-type surface.
An \textit{exiting sequence} is a collection 
$U_1 \supseteq  U_2 \supseteq \cdots $ of connected open subsets of $S$ such that
\begin{enumerate}[leftmargin=*]
	\item For any $n$, $U_n$ is not relatively compact.
	\item The boundary of $U_n$ is compact for all $n$.
	\item Every relatively compact subset of $S$ is disjoint from $U_n$ for all but
	finitely many $n$.
\end{enumerate}
Two exiting sequences are \textit{equivalent} if every element of the
first sequence is contained in some element of the second, and vice-versa.
An \emph{end of $S$} is an equivalence class of exiting sequences, and we
write $\mathrm{Ends}(S)$ for the space of ends of $S$.
Note that $\mathrm{Ends}(S)$ is a 
totally disconnected, separable and compact space. In particular, it is a subset of a Cantor set. We refer to \cite[Section 2.1]{bigsurvey} for the topology of the space of ends.
An end of $S$ is called \textit{planar} if it has a representative exiting sequence
whose elements are eventually planar; otherwise it is said to be \textit{non-planar} (or \emph{accumulated by genus}).
We denote by $\mathrm{Ends}^p(S)$ and $\mathrm{Ends}^{np}(S)$ the subspaces of planar
and non-planar ends of $S$, respectively. Clearly, $\mathrm{Ends}(S) = \mathrm{Ends}^p(S)\sqcup\mathrm{Ends}^{np}(S)$. 
 Infinite-type surfaces without boundary have been classified by the following (due to Ker\'{e}kj\'{a}rt\'{o} \cite{k23} and Richards \cite{richards}): 
two connected orientable surfaces $S_1$ and $S_2$ are homeomorphic if and only if they have the same genus,
and there exists a homeomorphism $h: 
\mathrm{Ends}(S_1)\rightarrow \mathrm{Ends}(S_2)$ whose restriction to $\mathrm{Ends}^{np}(S_1)$ defines a homeomorphism
between $\mathrm{Ends}^{np}(S_1)$ and $\mathrm{Ends}^{np}(S_2)$. Note that $\text{genus}(S)<\infty$ if and only if $\Ends^{np}(S)=\emptyset$.
\subsection{Mapping class groups}\label{mcg}
For a surface $S$, let $\Homeo(S)$ denote the group of homeomorphisms of $S$. The subgroup of $\Homeo(S)$ consisting of homeomorphisms which are isotopic to the identity is denoted by $\Homeo_0(S)$ (this is also the connected component of the identity with respect to the compact-open topology). Let $\Homeo^+(S)$ (resp. $\Homeo^-(S)$) denote the set of all orientation-preserving (resp. orientation-reversing) homeomorphisms of $S$. 
The (extended) \textit{mapping class group} of $S$, denoted by $\MCG^{\pm}(S)$, is the group of isotopy classes of homeomorphisms of $S$, i.e. $\MCG^{\pm}(S)=\Homeo(S)/\Homeo_0(S)$. Equivalently, it is defined as $\pi_0(\Homeo(S))$ with respect to the compact open topology. Now $\MCG^{\pm}(S)$ acts naturally by homeomorphisms on the space of ends $\Ends(S)$, and the kernel is called the pure mapping class group, denoted by $\PMCG^{\pm}(S)$. Both of these groups have subgroups of index $2$, denoted by $\MCG(S)$ and $\PMCG(S)$ respectively, consisting of orientation-preserving homeomorphisms. 
Notice that $\MCG(S)=\Homeo^+(S)/\Homeo_0(S).$ When the surface is of infinite-type, the mapping class group is homeomorphic to the Baire space, i.e. $\MCG(S)\cong \mathbb{R}\setminus\mathbb{Q}$. We call the mapping class group as \textit{big mapping class group} if the surface is of infinite-type. We define a particular type of element in $\MCG(S)$ which will be used later. Let $A=\mathbb{S}^1\times[0,1]$ be the annulus. Define the twist map $T:A\rightarrow A$ given by $T(\theta,t)=(\theta+2\pi t, t)$ for all $(\theta,t)\in \mathbb{S}^1\times[0,1]$. Let $\alpha$ be an essential simple closed curve in $S$ (i.e. $\alpha$ is an embedding of $\mathbb{S}^1$, not homotopic to a point or a puncture). Fix a regular neighbourhood $N$ of $\alpha$ and an orientation-preserving homeomorphism $\psi: A\rightarrow N$. Now define $T_{\alpha}: S\rightarrow S$ by 
 \[T_{\alpha}(s)= \left\{
\begin{array}{ll}
\psi\circ T\circ \psi^{-1}(s) & \text{if} \; s\in N \\
s & \text{if}\; s\in S\setminus N.
\end{array}\right. \] 
The map $T_{\alpha}$ is called a \emph{Dehn twist about $\alpha$}. Throughout the paper, we will use left Dehn twist.  Note that a Dehn twist neither depends on $N$ nor on $\alpha$ (up to isotopy). It is an infinite order element of $\MCG(S)$. Now we have the following basic result. For details, we refer the reader to \cite[Chapter 3]{fm12}. 
\begin{lemma}\label{dehnt}\cite[Section 3.3]{fm12} Let  $T_{\alpha}$ be a Dehn twist about $ \alpha$ and $f\in\MCG(S)$, and $m,n\in \N$. Then the following hold:
	\begin{enumerate}[leftmargin=*]
	\item\label{dehnt1} $fT_{\alpha}f^{-1}=T_{f(\alpha)}$. 
	\item\label{dehnt2} $T_{\alpha}^n=T_{\beta}^m$ if and only if $\alpha=\beta$ and $m=n$.
	\end{enumerate}

\end{lemma}
\subsection{Some basic results}

To determine the $R_{\infty}$-property we have to understand the automorphism group of $\MCG(S)$ (resp. $\MCG^{\pm}(S)$). Following the work of Bavard et. al. \cite{bdr17} we will briefly describe the automorphism group of big mapping class groups.

\noindent
\textbf{\textit{Automorphisms:}}
For a group $G$, let $\Aut(G)$ denote the group of all automorphisms of $G$, and $\Inn(G):=\{i_g\mid i_g(x)=gxg^{-1} \;\forall x\in G\}$, the normal subgroup of $\Aut(G)$ consisting of only the inner automorphisms of $G$. The outer automorphism group of $G$, is denoted by $\Out(G):=\Aut(G)/\Inn(G)$. A subgroup $H$ of $G$ is called \emph{characteristic} if $\varphi(H)=H$ for all $\varphi\in \Aut(G)$. 
 By virtue of \cite[Theorem 1.1]{bdr17}, for every $\varphi\in \Aut(\MCG^{\pm}(S))$ there exists $h\in\Homeo(S)$ such that $\varphi(f)=hfh^{-1}$, for all $f\in\MCG^{\pm}(S)$. That is every automorphism of $\MCG^{\pm}(S)$ is inner.  
Now let $\varphi\in \Aut(\MCG(S))$. Since $\MCG(S)$ is an index two subgroup of $\MCG^{\pm}(S)$ again by \cite[Theorem 1.1]{bdr17}, there exists $h\in\Homeo(S)$ such that $\varphi(f)=hfh^{-1}$. Now there are two possibilities for $h$: (a) $h\in\Homeo^+(S)$, or (b) $h\in\Homeo^{-}(S)$. \\
(a) If $h\in\Homeo^+(S)$, then $\varphi(f)=hfh^{-1}=i_{h}(f)$, i.e. $\varphi$ is an inner automorphism of $\MCG(S)$. 

\noindent
(b) If $h\in\Homeo^{-}(S)$, then 
	$\varphi^2(f)=\varphi(hfh^{-1})=h^2fh^{-2}=i_{h^2}(f)$. Thus $\varphi^2=i_{h^2}\in \Inn(\MCG(S))$. 

\noindent		
Therefore $\overline{\varphi}$ is an element of order $1$ or $2$ in $\Out(\MCG(S))$, where $\Aut(\MCG(S))\rightarrow \Out(\MCG(S))$ given by $\varphi\mapsto\overline{\varphi}$. Since $\varphi$ is an arbitrary automorphism, we have  $\Out(\MCG(S))\cong \Z/2\Z\cong\langle\overline{\psi}\mid \psi: f\mapsto hfh^{-1}, h\in\Homeo^{-}(S)\rangle$. We summarize this as follows:
\begin{lemma}\cite[Theorem 1.1, Corollary 1.2]{bdr17}\label{out}
If $S$ is an infinite-type surface then the following hold:
\begin{enumerate}[leftmargin=*]
\item\label{out1} $\Out(\MCG^{\pm}(S))=\{e\}$.
\item\label{out2} $\Out(\MCG(S))\cong \Z/2\Z$.
\item \label{out3} If $G<\MCG(S)$ is a finite index subgroup then $\Out(G)$ is finite.
\item\label{char1} $\MCG(S), \PMCG(S)$ and $\PMCG^{\pm}(S)$ are characteristic subgroups in $\MCG^{\pm}(S)$.
\end{enumerate}
\end{lemma}  
Suppose that $1\rightarrow N \overset{i}{\rightarrow} G\overset{\pi}{\rightarrow} Q\rightarrow 1$ is an exact sequence of arbitrary groups. If $Q$ is finite and $N$ has infinitely many conjugacy classes, then we \textbf{claim} that $G$ also has infinitely many conjugacy classes.
If $G$ has finitely many conjugacy classes then we will arrive at a contradiction. Since $N$ has infinitely many conjugacy classes there exist infinitely many elements $n_r\in N$ ($r\geq 0$) such that $n_r$ and $n_s$ are not conjugate in $N$ ($r\neq s$) but $n_0$ and $n_r$ are conjugate in $G$ for all $r\geq 0$. So for all $r\geq 1$ there exists $g_r\in G$ such that $n_r=g_rn_0g_r^{-1}$. 
Since $Q(\cong G/N)$ is finite there exist distinct $r,s\in \mathbb{N}$ such that $g_sg_r^{-1}\in N$. Therefore $n_s=g_sn_0g_s^{-1}=g_sg_r^{-1}n_rg_rg_s^{-1}=(g_sg_r^{-1})n_r(g_sg_r^{-1})^{-1}$. This is a contradiction to the assumption that $n_r$ and $n_s$ are not conjugate in $N$. We record this elementary observation as follows: 
\begin{lemma}\label{ntog}
Let $1\rightarrow N \rightarrow G\rightarrow Q\rightarrow 1$ be an exact sequence of groups. If $Q$ is finite and $N$ has infinitely many conjugacy classes then so does $G$.
\end{lemma}
\noindent
The above lemma holds in more generality (see, \cite[Lemma 2.2(ii)]{MS}). Next, we collect the following two important results: 
\begin{lemma}\cite[Lemma 2.1]{MS}\label{qtog}
	Let $1\rightarrow N \rightarrow G\rightarrow Q\rightarrow 1$ be an exact sequence of groups. Suppose that $N$ is a characteristic subgroup of $G$ and $Q$ has the $R_{\infty}$-property. Then $G$ also has the $R_{\infty}$-property.
\end{lemma}
\begin{lemma}\cite[Corollary 3.2]{FLT}\label{inner}
	Let $\varphi \in \Aut(G)$ and $i_g \in \mathrm{Inn}(G)$. Then $R(\varphi\circ i_g)=R(\varphi)$. In particular, $R(i_g)=R(\mathrm{Id})$, i.e. the number of inner twisted conjugacy classes in $G$ is equal to the number of conjugacy classes in $G$.
\end{lemma}
\noindent
\textbf{\textit{Some geometric results:}}
	Let $G$ be a group acting on a metric space $(X, d)$ by isometries. Let $x\in X$ and $g\in G$. We \textbf{claim} that   $\displaystyle\lim_{n\to\infty}\frac{d(x,g^nx)}{n}$ exists. Suppose that $f:\N \to [0,\infty)$ is a function given by $f(n)=d(x,g^nx)$. Clearly, $f(n+m) = d(x,g^{m+n}x) \leq d(x,g^mx)+d(x,g^nx)= f(m)+f(n)$, for any $m,n\in \N$. Thus, $f$ is subadditive. 
	Now, let $M=\inf \frac{f(n)}{n}$. Then for any $\epsilon>0$, there exists $n>0$ such that $\frac{f(n)}{n}<M+\epsilon$. So, for every $m>n$, we have $m=kn+r$, where $k\in \N$ and $0\leq r<n$. So, $\frac{f(m)}{m}=\frac{f(kn+r)}{m}\leq \frac{f(kn)}{m}+\frac{f(r)}{m} \leq \frac{kn(M+\epsilon)}{m}+\frac{f(r)}{m}$.
	Therefore $\lim_{n\to\infty}\frac{d(x,g^nx)}{n}=M$. This proves the claim. For instance, see \cite[Chapter II.6, Exercise 6.6(1)]{BH}.
	
	The \emph{asymptotic translation length of} $g\in G$ is given by $\tau(g):=\lim_{n\to\infty}\frac{d(x,g^nx)}{n}$. This in itself is of independent interest in geometric group theory. We have the following basic result: 
\begin{lemma}\label{translation-length}
	\begin{enumerate}[leftmargin=*]
		\item $\tau(g)$ does not depend on the choice of $x$.
		\item $\tau(g)=\tau(h^{-1}gh)$ for all $g,h \in G$.
		\item $\tau(g^k)=|k|\tau(g)$ for all $k\in \Z$.
	\end{enumerate}
\end{lemma}
\begin{proof}
\begin{enumerate}[leftmargin=*]
\item Let $y\in X$ be a point distinct from $x$. Then, we have\\ 
$d(x,g^nx)-d(x,y)-d(g^nx,g^ny)\leq d(y,g^ny)\leq d(y,x)+d(x,g^nx)+d(g^nx,g^ny)$.\\
This implies $\frac{d(x,g^nx)}{n}-2\frac{d(x,y)}{n}\leq \frac{d(y,g^ny)}{n}\leq\frac{d(x,g^nx)}{n}+2\frac{d(x,y)}{n}$.\\ 
Thus, we have $\lim_{n\to\infty}\frac{d(x,g^nx)}{n} =\lim_{n\to\infty}\frac{d(y,g^ny)}{n}$.
\item For any $g, h\in G$, we have $$\tau(h^{-1}gh)= \lim_{n\to\infty}\frac{d(x,(h^{-1}gh)^nx)}{n}=\lim_{n\to\infty}\frac{d(x,h^{-1}g^nhx)}{n}=\lim_{n\to\infty} \frac{d(hx,g^n(hx))}{n}.$$ 
Then by $(1)$, $\tau(h^{-1}gh)=\tau(g)$.
\item Assume that $k>0$, then $$\tau(g^k)=\lim_{n\to\infty}\frac{d(x,(g^k)^nx)}{n}=\lim_{n\to\infty}\frac{kd(x,g^{nk}x)}{nk}=k\lim_{nk\to\infty}\frac{d(x,g^{nk}x)}{nk}=k\tau(g).$$ Similarly one can check this for $k<0$. Hence we have the result.	\end{enumerate}
\end{proof}
\noindent
The above result immediately provides a sufficient condition for a group to have infinitely many conjugacy classes (cf. \cite[Corollary 8.2.G]{gromov}). Recall that an isometry $g$ is called \emph{loxodromic} if $\tau(g)>0$. 
\begin{corollary}\label{infinitecc}
	If a metric space $(X,d)$ admits an isometric $G$-action such that $\tau(g)>0$, then the number of conjugacy classes of $G$ is infinite.
\end{corollary}
\begin{proof}
First, note that the order of $g$ is infinite because $\tau(g)>0$. Now look at the infinite sequence $\{g^i\}_{i\in \mathbb{N}}$ of distinct elements of $G$. By \lemref{translation-length} $g^i$ is not conjugate to $g^j$, for $i\neq j$. Thus, $G$ has infinitely many conjugacy classes.
\end{proof}
\noindent
In order to describe the next set of results we fix some notations.
Let $\varphi\in\Aut(G)$ and $\Z=\langle t\rangle$. Define the following extension of $G$ containing $G$ as a normal subgroup:
$$G_{\varphi}:=\langle G,t\mid tgt^{-1}=\varphi(g)\;  \forall g\rangle=G\rtimes_{\varphi} \Z=\bigsqcup_{n\in\Z} G\cdot t^n.$$ 
Now, if $x$ and $y$ are $\varphi$-twisted conjugate in $G$ then $y=gx\varphi(g^{-1})$ for some $g\in G$. Therefore $y=gxtg^{-1}t^{-1}$ in $G_{\varphi}$ which implies that  $yt=g(xt)g^{-1}$, i.e. $xt$ and $yt$ are conjugate in $G_{\varphi}$.
Conversely, suppose that $xt$ and $yt$ are conjugate in $G_{\varphi}$. Then there exists an element $gt^n$ in $G_{\varphi}$ such that   $yt=(gt^n)(xt)(gt^n)^{-1}=g(t^nxt^{-n})tg^{-1}=g\varphi^n(x)tg^{-1}.$ Thus $y=g\varphi^n(x)(tg^{-1}t^{-1})=g\varphi^n(x)\varphi(g^{-1})$, i.e. $y\sim_{\varphi} \varphi^n(x)$. Also, note that $x$ and $\varphi(x)$ are always $\varphi$-twisted conjugate. Thus, $x\sim_{\varphi}y$.
We summarize this as follows:
\begin{lemma}\cite[Lemma 2.1]{fel10}\label{keylemma}
	Two elements $x$ and $y$ of $G$ are $\varphi$-twisted conjugate if and only if $xt$ and $yt$ of $G_{\varphi}$ are conjugate in $G_{\varphi}$. 
\end{lemma}
\noindent
This gives an immediate corollary which will be very useful later.
\begin{corollary}\label{bijection}
	The map $x\mapsto xt$ gives a bijection between the set of $\varphi$-twisted conjugacy classes of $G$ and the set of conjugacy classes of $G_{\varphi}$ contained in the coset $G\cdot t$. That is  $R(\varphi)$ is the number of conjugacy classes  of $G_{\varphi}$ in the coset $G\cdot t$ of $G$.
\end{corollary}
In a more general set-up we have the following result due to T. Delzant (cf. \cite[Lemma 6.3]{fg10}, \cite[Lemma 3.4]{ll00}). Let $G$ be a group acting on a hyperbolic metric space $X$ by isometries. The limit set $\Lambda(G)$ of $G$ is defined as $\Lambda(G):=\overline{Gx}\cap \partial X$ (for some $x\in X$), where $\partial X$ denotes the Gromov boundary of $X$. The action is called \emph{non-elementary} if $|\Lambda(G)|=\infty$. Note that if the action is non-elementary then there exists a loxodromic isometry. For details, see \cite[Section 8.2.D]{gromov}.
\begin{lemma}\cite[Lemma 3.4]{ll00}\label{delzant}
	Let $N$ be a normal subgroup of $G$ such that $G/N$ is abelian. If $G$ acts on a hyperbolic space non-elementarily, then any coset of $N$ contains infinitely many conjugacy classes of $G$.
\end{lemma}
\noindent
Now we recall the Alexander method for infinite-type surfaces.
\begin{lemma}\cite[Corollary 1.2]{hmv19}\label{alexander}
	Suppose that $S$ is an infinite-type surface. If $f\in \MCG^{\pm}(S)$ fixes each essential simple closed curve of $S$ then $f=\Id$.
\end{lemma}

\begin{lemma}\label{center}
	If $S$ is a connected infinite-type surface without boundary then the centers $Z(\MCG^{\pm}(S))=\{\Id\}=Z(\MCG(S))$.
\end{lemma}
\begin{proof}
	Lanier and Loving \cite[Proposition 2]{ll20} proved that $\MCG(S)$ is centerless. 
	In fact by the same line of argument, using the Alexander method, we prove that $Z(\MCG^{\pm}(S))=\{\Id\}$.
	If possible suppose that $\Id\neq f\in Z(\MCG^{\pm}(S))$. Then by \lemref{alexander} there exists a curve $\gamma$ such that $f(\gamma)\neq \gamma$. Now let $T_{\gamma}\in \MCG^{\pm}(S)$ be a Dehn twist about $\gamma$. Then $fT_{\gamma}f^{-1}=T_{\gamma}$ as $f$ is central. Therefore $T_{f(\gamma)}=T_{\gamma}$, which in turn implies that $f(\gamma)=\gamma$ (by \lemref{dehnt}). We get a contradiction to the fact that $f(\gamma)\neq \gamma$. Therefore the (extended) big mapping class group $\MCG^{\pm}(S)$ is also  centerless.
\end{proof}
\noindent
We end this section with the following result. 
Let $C$ be a Cantor space.
By a result of Rubin \cite[p. 492, (6)]{rubin}, every automorphism of $\Homeo(C)$ is inner\footnote{This result was brought to our notice by Yves de Cornulier}. Also, note that 
$\Homeo(C)$ has infinitely many (more precisely uncountably many) conjugacy classes (see \cite[Section 5, Theorem IV]{and58}). Therefore by \lemref{inner}, $\Homeo(C)$ has the $R_{\infty}$-property. We record this discussion as follows:
\begin{lemma}\label{cantorr}
	If $X$ is a topological space homeomorphic to the Cantor space, then $\Homeo(X)$ possesses the $R_{\infty}$-property.
\end{lemma}

\section{Main results}\label{results}
\noindent
We begin with the following basic result.
\begin{theorem}\label{mainthm1}
Let $S$ be a connected orientable infinite-type surface without boundary. Then the $($extended$)$ big mapping class group $\MCG^{\pm}(S)$ has the $R_{\infty}$-property.
\end{theorem}
\begin{proof}
Let $T_{\alpha}$ be a Dehn twist (with respect to an essential simple closed curve $\alpha$). Then $T_{\alpha}\in \MCG(S)$ is an infinite order element. Consider the infinite sequence $\{T_{\alpha}^n\mid n\in \N\}$ of distinct elements of $\MCG(S)$. For distinct $m,n\in \N$ if $T_{\alpha}^n$ is conjugate to $T_{\alpha}^m$ in $\MCG(S)$, then $fT_{\alpha}^nf^{-1}=T_{\alpha}^m$ for some $f\in \MCG(S)$. This implies that $T_{f(\alpha)}^n=T_{\alpha}^m$ (by \lemref{dehnt}\eqref{dehnt1}). Therefore, by \lemref{dehnt}\eqref{dehnt2} we have $f(\alpha)=\alpha$ and $m=n$, which is a contradiction. Hence $\MCG(S)$ has infinitely many conjugacy classes (of representatives $T_{\alpha}^n$ for each $n\in \N$). 
Now, look at the following short exact sequence of groups:
\[1\rightarrow\MCG(S)\rightarrow\MCG^{\pm}(S)\rightarrow\Z/2\Z\rightarrow 1.\]
Then by \lemref{ntog} $\MCG^{\pm}(S)$ has infinitely many conjugacy classes. Therefore by \lemref{inner} and \lemref{out}\eqref{out1} it follows that the (extended) big mapping class group $\MCG^{\pm}(S)$ possesses the $R_{\infty}$-property.
\end{proof}
\noindent
There is an important well-studied hyperbolic space (namely the curve graph $\mathcal{C}(S)$) on which the mapping class group of a finite-type surface $S$ acts non-elementarily. But for infinite-type surfaces there is no one such space on which a big mapping class group acts non-elementarily. So we shall deal with this case-by-case.
For the rest of the paper, we will only consider the big mapping class group $\MCG(S)$.
\subsection{Surfaces with non-displaceable subsurfaces}
A finite-type connected subsurface $\Sigma\subset S$ is called \textit{non-displaceable} if $f(\Sigma)\cap \Sigma\neq \emptyset$ for all $ f\in \Homeo(S)$.
In \cite[Theorem 3.13]{rafi}, the authors also gave an algebraic characterisation which says the following: $\Sigma\subset S$ is non-displaceable if and only if $\MCG(S)$ contains a non-trivial normal free subgroup. It follows from the classification of surfaces that there are uncountably many such infinite-type surfaces.
\begin{example}\label{ex}
	\begin{enumerate}[leftmargin=*]
		\item\label{finitegenus}
		Let $0<\text{Genus}(S)<\infty$. By the classification of surfaces in \cite{richards}, $S$ can be realized as $\Sigma-E$, where $E$ is a totally disconnected closed subset of $\Sigma$ homeomorphic to the space of ends $\Ends(S)$, satisfying $\text{Genus}(S)=\text{Genus}(\Sigma)$. 
		Then $\Sigma$ is a non-displaceable subsurface of $S$ as $\Sigma$ contains non-separating curves but $S-\Sigma$ does not. (See \cite[Example 2.4]{mr20}.)
		\item Let $S:=\Omega_n\;(n\geq 3)$ be an infinite genus surface with exactly $n$ non-planar ends. Then $S$ contains a non-displaceable subsurface.
	\end{enumerate} 
\end{example}
\noindent
For more examples, the reader is referred to \cite{mr20}. In order to state and prove one of our main theorems, we need another piece of information. There is a celebrated construction of quasi-trees of metric spaces (popularly known as \emph{BBF family for a group $G$}) developed by Bestvina, Bromberg and Fujiwara. We will not be describing this here but we refer the interested readers to \cite{BBF}. However, using this machinery, in \cite{rafi} the authors proved the following important result which will be sufficient for our purpose. 
\begin{lemma}\cite[Theorem 2.9]{rafi}\label{hqr}
Let $S$ be a connected orientable infinite-type surface which contains a connected non-displaceable subsurface $\Sigma$ of finite-type. Then there exists an unbounded hyperbolic space $X$  equipped with a $($continuous$)$ non-elementary isometric action of $\MCG(S)$ $($resp. $\MCG^{\pm}(S))$.
\end{lemma}
\begin{theorem}\label{mainthmnondis}
	Let $S$ be a connected orientable infinite-type surface without boundary which contains a non-displaceable subsurface. Then $\MCG(S)$ has the $R_{\infty}$-property.
\end{theorem}
\begin{proof}
 In view of \lemref{hqr}, the action of $\MCG(S)$ (resp. $\MCG^{\pm}(S)$) on $X$ is non-elementary. Then there is an element $g$ of $\MCG(S)$ such that $\tau(g)>0$ (i.e. $g$ is a loxodromic isometry with respect to this action). By \corref{infinitecc}, $\MCG(S)$ has infinitely many conjugacy classes (of loxodromic elements)\footnote{$\MCG(S)$ has infinitely many conjugacy classes by powers of a Dehn twist also as in \thmref{mainthm1}.}. Now let $\varphi\in \Aut(\MCG(S))$ which is not inner. Then by \lemref{out}\eqref{out2}, we have 
	$\MCG(S)_{\varphi}\cong\MCG(S)\rtimes_{\varphi}\Z/2\Z\cong\MCG^{\pm}(S)$, i.e. we have the following short exact sequence of groups
	$$1\rightarrow\MCG(S)\rightarrow \MCG(S)_{\varphi}\rightarrow \Z/2\Z\rightarrow 1,$$
	where $\Z/2\Z=\langle t\in \Homeo^{-}(S)\mid t^2=_{\MCG(S)}\Id\rangle$. 
	Therefore by \lemref{delzant}, $\MCG(S)$ contains infinitely many conjugacy classes of $\MCG^{\pm}(S)$.
	Hence by \corref{bijection}, $R(\varphi)=\infty$, which in turn implies that $\MCG(S)$ has the $R_{\infty}$-property.
\end{proof}
\begin{corollary}
With notations as in \thmref{mainthmnondis}, let $G<\MCG(S)$ be a finite index subgroup. Then $G$ has the $R_{\infty}$-property. 
\end{corollary}
\begin{proof}
Since $G$ is a finite index subgroup of $\MCG(S)$, it acts non-elementarily on the hyperbolic space $X$ (\lemref{hqr}). Thus, there exists a loxodromic element $g\in G$ such that $\tau(g)>0$. By \corref{infinitecc} $G$ has infinitely many conjugacy classes. Now by \lemref{out}\eqref{out3} $\Out(G)$ is finite. Hence the proof follows by virtue of \corref{bijection}.
\end{proof}
\subsection{Surfaces without non-displaceable subsurfaces}
From now on assume that an infinite-type surface $S$ does not contain a non-displaceable subsurface. Thus, in view of Example \ref{ex}\eqref{finitegenus}, genus of $S$ is either zero or infinity. Also, by \cite[Theorem 6.1]{apv} either the space of ends $\Ends(S)$ of $S$ is self-similar or it is doubly pointed. For example, the space of ends of Loch Ness Monster surface is self-similar and that of Jacob ladder surface is doubly pointed. In order to determine the $R_{\infty}$-property of such surfaces, we discuss the following general construction originally due to Birman.
\subsection{Birman exact sequence:}
Let $S$ be a connected orientable surface (both finite-type and infinite-type) without boundary. Fix $x\in S$. Then  $f\in \Homeo^+(S)$ evaluated at $x$ gives rise to the following \emph{fiber bundle}
$$\Homeo^+(S,x)\longrightarrow\Homeo^+(S)\overset{p}{\longrightarrow}S,$$ 
where $p(f)=f(x)$ and $\Homeo^+(S,x)$ is the set of all orientation preserving homeomorphisms of $S$ which fix $x$. Then it induces the following homotopy long exact sequence 
\begin{center}
	
	\begin{tikzpicture}[descr/.style={fill=white,inner sep=1.5pt}]
	\matrix (m)[matrix of math nodes,row sep=1em,column sep=2.4em,text height=1.5ex, text depth=0.25ex]
	{ \cdots & \pi_1(\Homeo^+(S,x)) & \pi_1(\Homeo^+(S)) & \pi_1(S,x) \\
		& \pi_0(\Homeo^+(S,x)) & \pi_0(\Homeo^+(S)) & \pi_0(S) & \cdots. \\};
	\path[overlay,->, font=\scriptsize,>=latex]
	(m-1-1) edge (m-1-2)
	(m-1-2) edge (m-1-3)
	(m-1-3) edge (m-1-4)
	(m-1-4) edge[out=355,in=175] node[descr,yshift=0.3ex] {} (m-2-2)
	(m-2-2) edge (m-2-3)
	(m-2-3) edge (m-2-4)
	(m-2-4) edge (m-2-5);
	\end{tikzpicture}
\end{center}

Note that the group $\pi_1(\Homeo^+(S))$ is trivial (if $S$ is of infinite-type see \cite[Theorem 1.1]{y00} and if $S$ is finite-type see \cite[Theorem 1.14]{fm12}), and $\pi_0(S)=\{e\}$ as $S$ is connected. Therefore we have
\begin{align}\label{birman}
1\longrightarrow\pi_1(S,x)\overset{\mathcal{P}}{\longrightarrow}\MCG(S,x)\overset{\mathcal{F}}{\longrightarrow}\MCG(S)\longrightarrow 1,
\end{align}
where the first map $\mathcal{P}$ is called a point pushing map and the last map $\mathcal{F}$ is the  forgetful map. Note that $\pi_1(S,x)\cong \pi_0(\Homeo^+(S,x)\cap\Homeo_0(S))$ (see \cite[Section 1.2.3]{mj11}). 
We refer the reader to \cite[p. 100, Section 4.2.3]{fm12} for more details of the above discussion. It is called  the \emph{Birman exact sequence}. Below we assume that $S$ is a surface with exactly one isolated puncture.

\begin{lemma}\label{char-lemma}
	The fundamental group $\pi_1(S,x)$ is a characteristic subgroup of $\MCG(S,x)$.
\end{lemma}
\begin{proof}
	As Equation \eqref{birman} is a short exact sequence, we have $\mathrm{ker}(\mathcal{F})=\mathrm{Im}(\mathcal{P}) =\pi_1(S,x)$. First, note that $\pi_1(S,x)$ is a normal subgroup, i.e. it is invariant under every inner automorphism of $\MCG(S,x)$.   Now let $\varphi\in\Aut(\MCG(S,x))$ which is not inner. By the discussion of \lemref{out} there exists an $h\in \Homeo^{-}(S,x)$ such that $\varphi(f)=hfh^{-1}$ for all $f\in \MCG(S,x)$. 
	Let $f\in \MCG(S,x)$ such that $\mathcal{F}(f)=\Id$. \\	
	\textbf{Claim:} $\varphi(f)\in \mathrm{ker}(\mathcal{F})$.\\	
	Let $\psi\in \Homeo^+(S,x)$ be an element representing $f$, i.e. $[\psi]=f$. Suppose  $\overline{\psi}$ is the image of $\psi$ in $\Homeo^+(S)$. Then $\overline{\psi}$ is isotopic to the identity homeomorphism of $S$, i.e. $[\overline{\psi}]=\Id$ since $f\in \mathrm{ker}(\mathcal{F})$.
	Now, $h\circ\psi\circ h^{-1}$ is an orientation preserving homeomorphism of $S$ fixing $x$ representing $h\circ f\circ h^{-1}$ (as both $h$ and $\psi$ fix $x$). Therefore $\overline{h\psi h^{-1}}$ is isotopic to the identity as $\overline{\psi}$ is homotopic to identity, i.e. $[\overline{h\psi h^{-1}}]=\Id$. Thus, $\varphi(f)\in \mathrm{ker}(\mathcal{F})$.
\end{proof}
\begin{corollary}\label{stosx}
	If $\MCG(S)$ has the $R_{\infty}$-property then $\MCG(S,x)$ also has the $R_{\infty}$-property.  
\end{corollary}
\begin{proof}
	This follows from Equation \eqref{birman}, \lemref{qtog}, and \lemref{char-lemma}.
\end{proof}
\noindent
Now, let $C$ denote the Cantor set.
\begin{lemma}\label{s2-c}
	Then $\MCG(\mathbb{S}^2\setminus C)$ has the $R_{\infty}$-property.
\end{lemma}
\begin{proof}
	Observe that $\Ends(\mathbb{S}^2\setminus C)$ is homeomorphic to $C$. The mapping class group $\MCG(\mathbb{S}^2\setminus C)$ acts on its space of ends $C$ by homeomorphisms. This in turn induces the following short exact sequence of groups 
	\[1\rightarrow\PMCG(\mathbb{S}^2\setminus C)\rightarrow\MCG(\mathbb{S}^2\setminus C)\rightarrow\Homeo(C)\longrightarrow 1.\]
Now, in view of \lemref{cantorr}, $\Homeo(C)$ possesses the $R_{\infty}$-property. As $\PMCG(\mathbb{S}^2\setminus C)$ is a characteristic subgroup of $\MCG(\mathbb{S}^2\setminus C)$, applying \lemref{qtog}, $\MCG(\mathbb{S}^2\setminus C)$ has the $R_{\infty}$-property. 
\end{proof}
\begin{corollary}\label{r2-c}
	The mapping class group $\MCG(\R^2\setminus C)$ has the $R_{\infty}$-property.
\end{corollary}
\begin{proof}
	Look at the following Birman exact sequence of groups (cf. Equation \eqref{birman}) 
	\[1\rightarrow\pi_1(\mathbb{S}^2\setminus C, x)\overset{\mathcal{P}}{\rightarrow}\MCG(\mathbb{R}^2\setminus C)\overset{\mathcal{F}}{\rightarrow}\MCG(\mathbb{S}^2\setminus C)\rightarrow 1.\] 
 Now note that $\pi_1(\mathbb{S}^2\setminus C, x)$ is a characteristic subgroup $\MCG(\R^2\setminus C)$ (by \lemref{char-lemma}). By \lemref{s2-c}, $\MCG(\mathbb{S}^2\setminus C)$ has the $R_{\infty}$-property. Hence the result follows from \lemref{qtog}.  
\end{proof}
\begin{remark}
	By a result of Bavard, $\MCG(\R^2\setminus C)$ acts non-elementarily on a hyperbolic space called the ray graph. Thus, the $R_{\infty}$-property of $\MCG(\mathbb{R}^2\setminus C)$ also follows by \corref{bijection} and \lemref{delzant}.
\end{remark}
\noindent
We end this section with the following important class of groups which has the $R_{\infty}$-property.
\begin{theorem}\label{111}
Let $S$ be an infinite-type surface such that the space of ends $\Ends(S)$ is a Cantor set $C$ with the following: either (a) $\mathrm{genus}(S)<\infty$ or (b) $\Ends(S)=\Ends^{np}(S)$.
Then $\MCG(S)$ possesses the $R_{\infty}$-property.
\end{theorem}
\begin{proof}
If $\mathrm{genus}(S)<\infty$, then $\Ends^{np}(S)=\emptyset$. So in both cases we get the following short exact sequence of groups (cf. \cite[Section 2.4.1]{bigsurvey})
\[1\rightarrow\PMCG(S)\rightarrow\MCG(S)\rightarrow\Homeo(C)\longrightarrow 1.\] By \lemref{cantorr}, $\Homeo(C)$ satisfies the $R_{\infty}$-property. Hence by \lemref{qtog}, $\MCG(S)$ satisfies the $R_{\infty}$-property since $\PMCG(S)$ is a characteristic subgroup.
\end{proof}
\begin{remark}
For example, the Blooming Cantor tree surface ($\mathrm{genus}(S)=\infty$ and $\Ends(S)=\Ends^{np}(S)=C$) and the Cantor tree surface ($\simeq\mathbb{S}^2\setminus C$ and $\mathrm{genus}(S)=0$) satisfy the hypothesis of the previous theorem. See \cite{bigsurvey}, for details.
\end{remark}
\section{Isogredience classes}\label{isogred}
Most of what is contained in this section are known and can mainly be found in \cite{ft15}, \cite{bdr21} and the references therein. Still, it seems worthwhile to combine several results in a fashion that suits our purpose and make the paper more self-contained. 
Suppose $\Phi\in \Out(G):=\Inn(G)\backslash\Aut(G)$. Two elements $\alpha, \beta\in \Phi$ are said to be \emph{isogredient} (or similar) if $\beta=i_h\circ \alpha \circ i_{h^{-1}}$ for some $h\in G$, where $i_h(g)=hgh^{-1}$ for all $g\in G$. Observe that  this is an equivalence relation on $\Phi$. Fix a representative $\gamma\in \Phi$. Then $\alpha=i_a\circ \gamma$ and $\beta=i_b\circ \gamma$ for some $a,b\in G$. Therefore 
\begin{align*}
i_b\circ \gamma&=\beta=i_h\circ \alpha \circ i_{h^{-1}}
=i_h\circ i_a\circ \gamma \circ i_{h^{-1}}\\
i_b&=i_h\circ i_a\circ (\gamma \circ i_{h^{-1}}\circ\gamma^{-1})\\
i_b&=i_h\circ i_a\circ i_{\gamma(h^{-1})}\\
i_b&=i_{ha\gamma(h^{-1})}.
\end{align*}
Thus, $\alpha$ and $\beta$ are isogredient if and only if $b=ha\gamma(h^{-1})c$ for some $c\in Z(G)$, the center of $G$.
Let $S(\Phi)$ denote the number of isogredience classes of $\Phi$. If $\Phi=\overline{\textup{Id}}$, then $S(\overline{\textup{Id}})$ is the number of ordinary conjugacy classes of $G/Z(G)$. A group $G$ has the \emph{$S_{\infty}$-property} if $S(\Phi)=\infty$ for all $\Phi\in \Out(G)$. Observe that if a group $G$ has the $S_{\infty}$-property then $G$ has the $R_{\infty}$-property. We summarize this discussion as follows:
\begin{lemma}\cite[cf. Theorem 3.4]{ft15}\label{rvss}
	\begin{enumerate}[leftmargin=*]
	\item\label{1} If $G$ has the $S_{\infty}$-property then $G$ has the $R_{\infty}$-property.
	\item\label{2}  If $G$ is centerless and $G$ has the $R_{\infty}$-property then $G$ has the $S_{\infty}$-property.
	\end{enumerate}
	
\end{lemma} 
\noindent 
Levitt and Lustig \cite[Theorem 3.5]{ll00} proved that any non-elementary hyperbolic group has the $S_{\infty}$-property. A torsion-free non-elementary hyperbolic group is centerless, therefore it satisfies the $S_{\infty}$-property if and only if it satisfies the $R_{\infty}$-property. 
\begin{remark}
	Before these, Cohen and Lustig \cite[Proposition 5.4]{cl89} studied the $S_{\infty}$-property for a free group $\mathbb{F}_n$ although they did not use the terminology $S_{\infty}$. Best to our knowledge, the term `group with the $S_{\infty}$-property' was introduced by Fel'shtyn and Troitsky in \cite{ft15}, especially see Lemma 3.3 and Theorem 3.4 for the interplay between the $R_{\infty}$-property and $S_{\infty}$-property.
\end{remark}
\noindent
Now we are in a position to prove the main result of this section.
\begin{theorem}\label{mainthm2}
	Let $S$ be a connected orientable infinite-type surface without boundary. Then we have the following:
	\begin{enumerate}[leftmargin=*]
	\item the $($extended$)$ big mapping class group $\MCG^{\pm}(S)$ possesses the $S_{\infty}$-property.
	\item If $S$ contains a connected non-displaceable subsurface of finite-type then $\MCG(S)$ has the $S_{\infty}$-property.
	\item $\MCG(S)$ has the $R_{\infty}$-property if and only if $\MCG(S)$ has the $S_{\infty}$-property.
	\end{enumerate}
\end{theorem}
\begin{proof}
	\begin{enumerate}[leftmargin=*]
	\item In view of \thmref{mainthm1}, $\MCG^{\pm}(S)$ has the $R_{\infty}$-property. Now by virtue of \lemref{center} we have $Z(\MCG^{\pm}(S))=\{\Id\}$. Therefore by \lemref{rvss}\eqref{2} $\MCG^{\pm}(S)$ possesses the $S_{\infty}$-property.
	\item By \thmref{mainthmnondis} $\MCG(S)$ satisfies the $R_{\infty}$-property. By \lemref{center}, the center is trivial, i.e.  $Z(\MCG(S))=\{\Id\}$. Therefore by \lemref{rvss}\eqref{2}, $\MCG(S)$ has the $S_{\infty}$-property.
	\item This follows from \lemref{rvss} and \lemref{center}. 
	\end{enumerate}
\end{proof}
\begin{remark}
If $S$ is a finite-type surface without boundary (except few cases) then we know that the center $Z(\MCG(S))=\{\Id\}$ (see \cite[Theorem 3.10]{fm12}). Now $\MCG(S)$ satisfies the $R_{\infty}$-property by \cite[cf. Theorem 6.4, Theorem 4.3]{fg10}. So, again  in view of \lemref{rvss}, $\MCG(S)$ has the $S_{\infty}$-property.
\end{remark}
\noindent
We would like to end this paper with the following problem.

\medskip
\noindent
\textbf{Problem:} Suppose that $S$ is either the Loch Ness Monster surface or the Jacob Ladder surface. Does $\MCG(S)$ satisfy the $R_{\infty}$-property? 

\medskip
\noindent
\textbf{Acknowledgement:} The authors would like to thank Pranab Sardar for bringing the notion of infinite-type surfaces to our notice and for his encouragement throughout this work. We thank Anirban Bose for some useful discussion. 
We also thank the reviewer for crucial comments and suggestions which improved the exposition considerably.


\end{document}